\documentclass[a4paper,12pt]{article}
\usepackage{amsmath}
\usepackage{amssymb}
\usepackage{setspace}
\usepackage{fullpage}
\usepackage{bbm}
\usepackage{amsthm}
\usepackage[latin1]{inputenc}
\usepackage{xcolor}
\usepackage[normalem]{ulem}
\usepackage{hyperref}
\usepackage{rotating}
\usepackage{graphicx}
\usepackage{enumerate}
\usepackage{dsfont}
\usepackage{pdfpages} 
\usepackage{mathrsfs}

\usepackage[normalem]{ulem} 

\usepackage[
singlelinecheck=false 
]{caption}

\newtheorem{theorem}{Theorem}[section]
\newtheorem{lemma}[theorem]{Lemma}
\newtheorem*{rtheorem}{Main Theorem}
\newtheorem{corollary}[theorem]{Corollary}
\newtheorem{result}[theorem]{Result}
\newtheorem{proposition}[theorem]{Proposition}

\theoremstyle{definition}

\newtheorem{construction}[theorem]{Construction}
\newtheorem{definition}[theorem]{Definition}

\newtheorem{remark}[theorem]{Remark}
\newtheorem*{remark*}{Remark}

\newtheorem{example}[theorem]{Example}
\newtheorem*{example*}{Example}

\usepackage{pdfsync}
\def\PG{\mathrm{PG}}

\def\v{\mathbf{v}}
\def\u{\mathbf{u}}
\def\w{\mathbf{w}}
\def\e{\mathbf{e}}

\def\F{\mathbb{F}}

\title{The geometric field of linearity of linear sets}

\author{Dibyayoti Jena \and Geertrui Van de Voorde\thanks{Both authors are supported by the Marsden Fund Council administered by the Royal Society of New Zealand.}}
\date{}

\begin{document}
\maketitle
	
\begin{abstract}
If an $\F_q$-linear set $L_U$ in a projective space is defined by a vector subspace $U$ which is linear over a proper superfield of $\F_{q}$, then all of its points have weight at least $2$. 
It is known that the converse of this statement holds for linear sets of rank $h$ in $\PG(1,q^h)$ but for linear sets of rank $k<h$, the converse of this statement is in general no longer true. 

The first part of this paper studies the relation between the weights of points and the size of a linear set, and introduces the concept of the {\em geometric field of linearity} of a linear set. This notion will allow us to show the main theorem, stating that for particular linear sets without points of weight $1$, the converse of the above statement still holds as long as we take the geometric field of linearity into account.


\end{abstract}

\section{Sets and their linearity}
\subsection{Linear sets}

	Linear sets in finite projective spaces are a generalisation of subgeometries. They have attracted a lot of attention in the last few years. Their applications include construction of other mathematical objects like blocking sets \cite{polito}, translation ovoids \cite{lunardon}, KM-arcs \cite{maarten}. They have also proven useful in the study of semifields \cite{olgamichel} and rank metric codes \cite{zullo,jon,zini}. More details about linear sets can be found in \cite{wij,olga}.   

	More formally, let $\mathbb{F}_{q^h}$ denote the finite field of order $q^h$, where $q$ is a prime power, and let $V$ be an $r$ dimensional vector space over $\mathbb{F}_{q^h}$. The vector space $V$ can also be seen as a $rh$ dimensional vector space over $\mathbb{F}_q$. An $\mathbb{F}_q$-\textit{linear set $L$ of rank $k$} in $\PG(V)=\PG(r-1,q^h)$ is the set of points defined by a a $k$-dimensional $\mathbb{F}_q$-subspace $U$ of $V$ in the following way: 
	$$L=L_U=\{\langle u\rangle_{q^h}\mid u\in U\backslash\{0\}\}.$$
	Here $\langle u\rangle_{q^h}$ denotes the point of $\PG(r-1,q^h)$ determined by the vector $u$ (where $u$ is seen as a vector in $V$). If $W$ is an $\F_{q^h}$-vector subspace of dimension $s$ of $V$, then $\langle W\rangle_{q^h}$ denotes the projective $(s-1)$-dimensional subspace corresponding to $U$ in $\PG(V)$. To avoid confusion, the subspace spanned by two subspaces, say $S_1,S_2$, of $\PG(r-1,q^h)$ will be denoted by $span(S_1,S_2)$, and likewise, the vector subspace spanned by two vector spaces $V_1,V_2$ will be denoted by $span(V_1,V_2)$. 
	
	The {\em $\F_q$-weight} of a point $P=\langle u_P\rangle_{q^h}$ in a linear set $L_U$ is defined to be the vector space dimension of the $\mathbb{F}_q$-subspace $U_P$ of all the vectors determining $P$, i.e. $$U_P=\{0\}\cup\{u\mid u\in U, \langle u\rangle_{q^h}=\langle u_P\rangle_{q^h}\}.$$ 
	
We will simply use the term {\em weight} of a point if the underlying field is clear. Note that we need to specify the underlying vector space $U$ in the definition of the rank of a linear set and in the definition of the weight of a point in a linear set $L_U$; this will become more clear later in this paper.
	
	The isomorphism between $V=\mathbb{F}_{q^h}^r$ and $V=\mathbb{F}_q^{rh}$ induces a natural map $\phi$, called the \textit{field reduction map}, from $\PG(r-1,q^h)$ to $\PG(rh-1,q)$. The map $\phi$ takes points to $(h-1)$-dimensional subspaces, and in general $(n-1)$-dimensional subspaces in $\PG(r-1,q^h)$ to $(nh-1)$-dimensional subspaces of $\PG(rh-1,q)$. The images of points under $\phi$ form a \textit{Desarguesian} $(h-1)$-spread $\mathcal{D}$ of $\PG(rh-1,q)$, and every subfield $\F_{q^s}$ of $\F_{q^h}$ gives rise to a unique Desarguesian $(s-1)$-spread partitioning the elements of $\mathcal{D}$. Rank $k$ linear sets $L_U$ can be viewed geometrically as the pre-image under $\phi$ of sets of elements of $\mathcal{D}$ intersecting a fixed $(k-1)$-dimensional projective subspace $\langle U\rangle_q$ of $\PG(rh-1,q)$. The $\F_q$-weight of a point $P$, from this point of view, is one more than the dimension of $\phi(P)\cap\langle U\rangle_q$.
	It is well known (and it follows easily from the above point of view) that a Desarguesian spread is {\em normal}: if $M,N\in \mathcal{D}$, then the subspace spanned by them, $\langle M,N\rangle$ is partitioned by elements of $\mathcal{D}$.

	The original point of view on linear sets was given by \cite{lunardon} using {\em projections}. Let $\Sigma^*=\PG(k-1,q^h)$, $\Sigma=\PG(k-1,q)$ be a canonical subgeometry of $\Sigma^*$, and $\Omega=\PG(r-1,q^h)$ be a $(r-1)$-dimensional subspace of $\Sigma^*$. For a $(k-r-1)$-dimensional subspace $\Pi$ of $\Sigma^*$ disjoint from $\Sigma$ and $\Omega$, the \textit{projection map} $\mathscr{P}_{\Pi,\Omega}:\Sigma\rightarrow\Omega$ is defined by $\mathscr{P}_{\Pi,\Omega}(x)=\langle x,\Pi\rangle\cap\Omega$. The image of $\Sigma$ under  $\mathscr{P}_{\Pi,\Omega}$ is the {\em projection} of $\Sigma$ from $\Pi$ onto $\Omega$. It was proven in \cite[Theorems 1 and 2]{lunardon} that any linear set $L$ of rank $k$ in $\Omega=\PG(r-1,q^h)$ is either a canonical subgeometry of $\Omega$ or equivalent to a projection of $\Sigma=\PG(k-1,q)$, a canonical subgeometry of $\Sigma^*=\PG(r-1,q^h)$, from $\Pi$ to $\Omega$, where $\Pi$ is a $(k-r-1)$-dimensional subspace of $\Sigma^*$ disjoint from $\Sigma$ and $\Omega$. From this point of view, the $\F_q$-weight of a point $P$ is one more than the dimension of the pre-image of $P$ under the projection map (see also \cite[Proposition 2.7]{jon}).

	


	\subsection{The geometric field of linearity of a linear set}
	In this paper, we will distinguish between sets {\em being $\F_{q^s}$-linear sets}, and having {\em geometric field of linearity} $\F_{q^s}$. We will explain the reason for this distinction in detail but start with the definitions as used in this paper.
	\begin{definition} An $\F_q$-linear set $L_U$, defined by an $\F_q$-vector space  $ U$, is {\em an $\F_{q^s}$-linear set} if $U$ is an $\F_{q^s}$-vector space (i.e. if the set of vectors in $U$ is also closed under taking $\F_{q^s}$-multiples.) \end{definition}
	
	\begin{definition} A {\em strictly} $\F_{q^s}$-linear set is an $\F_{q^s}$-linear set that is not an $\F_{q^i}$-linear set for any $i>s$. 	
	\end{definition}
	We also say (as defined in \cite{classes}) that the {\em maximum field of linearity} of a strictly $\F_{q^s}$-linear set $L_U$ is $\F_{q^s}$.
	
	We introduce the following definitions in this paper:
	\begin{definition} An $\F_q$-linear set $L_U$ {\em has geometric field of linearity $\F_{q^s}$} if there exists an $\F_{q^s}$-linear set $L_W$ with $L_U=L_W$.
	\end{definition}
	
	\begin{definition} An $\F_q$-linear set $L_U$ {\em has maximum geometric field of linearity $\F_{q^s}$} if $s$ is the largest integer such that $L_U$ has geometric field of linearity $\F_{q^s}$.
	\end{definition}
	
	 It follows from the definitions that the order of the maximum geometric field of linearity is at least the order of the maximum field of linearity of that set.
	
	In many of the well-studied cases, there is no distinction between linear sets being $\F_{q^s}$-linear and having geometric field of linearity $\F_{q^s}$. In particular, we will see in Proposition \ref{irrelevant} that this is the case for linear sets of rank $h$ in $\PG(1,q^h)$. 
		For {\em simple} linear sets, the same holds true as seen in the following Remark.
		
		\begin{remark} A {\em simple} $\F_q$-linear set $L_U$ in $\PG(1,q^h)$ (or an $\F_q$-linear set of {\em class $1$} as introduced in \cite{classes}) has the property that if $L_U=L_V$, then $U=\lambda V$ for some $\lambda\in \F_{q^h}^*$. Since $U$ is $\F_{q^s}$-linear if and only if $\lambda U$ is $\F_{q^s}$-linear, it immediately follows from the definition that for simple $\F_q$-linear sets, the maximum field of linearity and the maximum geometric field of linearity coincide.
	\end{remark}


The proof of Proposition \ref{irrelevant} is essentially a corollary of the following theorem which discusses the number of directions determined by a function over a finite field. Note that the {\em graph} of a function $f$ is defined as the subspace $\{(x,f(x))\mid x\in \F_{q^h}\}$ of $(\F_{q^h})^2\cong \mathrm{AG}(2,q^h)$.

	\begin{theorem}[\cite{simeon,simeon2}]\label{th:graph1}
Let $f:\F_{q_0} \rightarrow \F_{q_0}$ be a function, $q_0$ a power of the prime $p$. Let $N$ be the number of directions determined by $f$. Let $r = p^s$ 
be maximal such that any line with a direction determined by $f$ that is incident with a point of the graph of 
$f$ is incident with a multiple of $r$ points of the graph of $f$. Then one of the following holds:
\begin{enumerate}
\item $r = 1$ and $\frac{q_0+3}{2} \leq N \leq q_0+1$;
\item $\F_r$ is a subfield of $\F_{q_0}$ and $\frac{q_0}{r} + 1 \leq N \leq \frac{q_0-1}{r-1}$;
\item $r=q_0$ and $N=1$.
\end{enumerate}
Moreover, if $r > 2$, then the graph of $f$ is $\F_r$-linear.
\end{theorem}

\begin{remark} The previous theorem says that if $r>2$, the point set of the graph of $f$ is $\F_r$-linear, which the author shows by showing that $f$ is an $\F_r$-linear map. It is clear that $\{(x,f(x))\mid x\in \F_{q^h}\}$ is indeed $\F_r$-linear if and only if $f$ is an $\F_r$-linear map.
\end{remark}


	In Proposition 2.2 of \cite{classes}, the following is shown, using Theorem \ref{th:graph1}:
	\begin{result}\label{prop22} Let $L_U$ be an $\F_q$-linear set of $\PG(W,q^h)=\PG(1,q^h)$ of rank $h$. The maximum field of linearity of $L_U$ is $\F_{q^s}$ where $s=\min\{w_{L_U}(P): P\in L_U\}$. 
	
		If the maximum field of linearity is $\F_q$, then the rank of $L_U$ as an $\F_q$-linear set is uniquely defined, i.e. for each $\F_q$-subspace $V$ of $W$, if $L_V=L_W$, then $\dim_q(V)=h$.
	\end{result}
	
	The proof of the following proposition goes along the same lines as the proof of Result \ref{prop22}.
	\begin{proposition}\label{irrelevant} Let $L_U$ be an $\F_q$-linear set of rank $h$ in $\PG(1,q^h)$. Then $L_U$ is a strictly $\F_{q^s}$-linear set if and only if $L_U$ has maximum geometric field of linearity $\F_{q^s}$.
	\end{proposition}

	\begin{proof}

	Let $L_U$ be an $\F_q$-linear set of rank $h$ in $\PG(1,q^h)$, where $U$ is strictly $\F_{q^s}$-linear. Every $\F_q$-linear set of rank $h$ in $\PG(1,q^h)$ can be mapped by an element of $\mathrm{PGL}(1,q^h)$ to a linear set not containing the point $\langle (0,1)\rangle_{q^h}$, and hence, be written as a set of points $L_U$ with $U=\{(x,f(x))\mid x\in \F_{q^h}\}$ for some $\F_q$-linear map $f$ which is strictly $\F_{q^s}$-linear.
	
 Since $U$ is strictly $\F_{q^s}$-linear we have that every vector line through two vectors of $U$ contains a multiple of $q^s$ points of $U$ but there is no $i>s$ such that every vector line through two vectors of $U$ contains a multiple of $q^{i}$ points of $U$. By Theorem \ref{th:graph1} we have that the number of directions determined by $f$, and hence, the number of points in $L_U$ is contained in $[q^{h-s}+1,\frac{q^h-1}{q^s-1}]$. 
 	Assume that there exists an $L_V=L_U$ such that $V$ is $\F_{q^{s'}}$-linear for some $s'>s$.

 If the rank of $V$ is $h$, then $L_V=\{\langle (x,g(x))\rangle_{q^h}\mid x\in \F_{q^h}\}$ for some $\F_{q^{s'}}$-linear map $g$. Similarly, by Theorem \ref{th:graph1} the number of points in $L_V$ is then contained in $[q^{h-s'}+1,\frac{q^h-1}{q^{s'}-1}]$. It follows that $s=s'$, a contradiction. If the rank of $V$ is larger than $h$, then $L_V$ is the entire line $\PG(1,q^h)$, a contradiction.
	Now assume that $V$ has rank smaller than $h$. Since $V$ is $\F_{q^{s'}}$-linear, the rank of $V$ is a multiple of $s'$, so it follows that the rank of $V$ is at most $h-s'$. This means that $L_V$ has at most $\frac{q^{h-s'}-1}{q-1}$ points, and we know that the number of points in $L_U$ is contained in $[q^{h-s}+1,\frac{q^h-1}{q^s-1}]$. Since $s'>s$, this is a contradiction.
	This argument shows that if an $\F_q$-linear set $L_U$ of rank $h$ is strictly $\F_{q^s}$-linear, its maximum field of linearity is $\F_{q^s}$.
	
	Vice versa, suppose that $L_U$ is an $\F_{q}$-linear set of rank $h$ with maximum field of linearity $\F_{q^s}$, that it, such that there exists some $L_V=L_U$ where $V$ is $\F_{q^s}$-linear. The above argument shows that if $L_V$ has rank $h$, then $U$ is indeed $\F_{q^s}$-linear. So suppose that $L_V$ has rank smaller than $h$, then we know that the rank of $V$ is at most $h-s$ and the number of points in $L_V$, and hence, $L_U$ is at most $\frac{q^{h-s}-1}{q-1}$. Since $U$ has rank $h$, this implies, as above, that $U$ is (at least) $\F_{q^s}$-linear. Since the maximum field of linearity of $L_U$ is $\F_{q^s}$ it follows that $U$ is strictly $\F_{q^s}$-linear.	
	\end{proof}
The following example describes a strictly $\F_q$-linear set with a geometric field of linearity different from $\F_q$, showing that the notions of being $\F_{q^s}$-linear and having geometric field of linearity $\F_{q^s}$ do not always coincide.

\begin{example} \label{deelrechte}Consider the set $S$ of points of the subline $\PG(1,q^3)$ contained in $\PG(1,q^{6})$. The set $S$ equals $L_U$ where $U=\{(a,b)\mid a,b \in \F_{q^3}\}$ and as such, we see that $S=L_U$ is an $\F_{q^3}$-linear set of rank $2$ over $\F_{q^3}$ and an $\F_q$-linear set of rank $6$. 
				
		Now consider any vector subspace $V$ of $U$ of dimension $5$. Then $L_U=L_V$ and $L_V$ is an $\F_q$-linear set of rank $5$. 
		
		Now $V$ is not $\F_{q^i}$-linear for any $i>1$, and hence, $L_V$ is a strictly $\F_q$-linear set. But it is clear that $L_V$ behaves as an $\F_{q^3}$-linear set (as it is in fact simply a subline $\PG(1,q^3)$). And indeed, according to our definition, $L_V$ has geometric field of linearity $\F_{q^3}$.

		\end{example}

	\begin{remark}Care needs to be taken when using the terminology on the linearity of linear sets as used by Sziklai in \cite{blockingsets}: we believe that the maximum field of linearity as used by the author should correspond to our definition of maximum geometric field of linearity.
	In particular, the author conjectures in \cite{blockingsets} that a linear set with rank $t+1$ in $\PG(2,p^h)$ with ``maximum field of linearity $\F_{p^e}$'' (between quotation marks) has at least $(p^e)^t+(p^e)^{t-1}+1$ points. 	
	
	Now consider,  similar to Example \ref{deelrechte}, the point set  $L_U$ in $\PG(2,q^4)$ defined by the $\F_{q^2}$-vector space $U=\{(a,b,c)\mid a,b,c\in \F_{q^2}\}$. We see that $L_U$ determines a Baer subplane $\PG(2,q^2)$ and has $q^4+q^2+1$ points. Now $L_U=L_V$ with 
		 $V=\{(a,b,c)\mid a\in \F_q,b,c\in \F_{q^2}\}$ where $L_V$ has rank $5$ and has maximum field of linearity $\F_q$.
		 Since $q^4+q^2+1<q^4+q^3+1$ this example violates the lower bound predicted by the above conjecture. However, this linear set has geometric field of linearity $\F_{q^2}$, and we believe this is what the author intended when writing the ``maximum field of linearity''.
		 Elsewhere in the paper, the {\em (geometric) exponent} of a point $P$ (in a blocking set $B$ in $\PG(2,p^h)$) is defined as the maximum integer $e$ such that all lines through $P$ meet the set $B$ in $1\pmod{p^e}$ points. Furthermore, it is shown that for a blocking set where all points have exponent $e$, $\F_{p^e}$ is a subfield of $\F_{p^h}$ and conjectured that the set $B$ is an $\F_{p^e}$-linear set. Note that in our example above, every line through a point $P$ of $L_V$ meets $L_V$ in $1\pmod{p^2}$ points, again indicating that $\F_{p^2}$ is the intended ``maximum field of linearity'', rather than the field $\F_q$ that follows from the definition. We choose the terminology for {\em geometric} field of linearity for its similarity with the distinction Sziklai makes between the {\em algebraic} and {\em geometric} exponent of a point set.

	\end{remark}

	\subsection{The minimum size of a linear set and points of weight at least two}
	We now link the field of linearity of a linear set with the weights of its points. The following lemma is easy to see.
	
	\begin{lemma}\label{triv} Let $L_U$ be an $\F_{q^s}$-linear set, then all points have $\F_q$-weight at least $s$.
	\end{lemma}
	\begin{proof} The weight of $P$ is the $\F_q$-dimension of the space $$U_P=\{0\}\cup\{u\mid u\in U, \langle u\rangle_{q^h}=\langle u_P\rangle_{q^h}\}.$$ 
	
	 Since $U$ is closed under $\F_{q^s}$-multiplication, for any $u\in U$, also $\beta u\in U$ where $\beta\in \F_{q^s}$, and  since $\langle u\rangle_{q^h}=\langle \beta u\rangle_{q^h}$, we see that $U_P$ is at least $s$-dimensional.
	\end{proof}

	The converse of Lemma \ref{triv} is not true as seen from the example below.
		\begin{example} \label{deelrechte2}Consider again Example \ref{deelrechte}. We see that $S=L_U$ is an $\F_{q^3}$-linear set of rank $2$ over $\F_{q^3}$ and an $\F_q$-linear set of rank $6$. Every point of $L_U$ has weight $3$, as predicted by Lemma \ref{triv}.
				Every point of $L_V$ still has weight at least $2$, but $L_V$ is not an $\F_{q^2}$-linear, nor $\F_{q^3}$-linear set (since $V$ is a strictly $\F_q$-linear vector space of dimension $5$). However $L_V$ has geometric field of linearity $\F_{q^3}$. The main theorem of this paper will show that this behaviour is not a coincidence.		
		 
		\end{example}

		The research of this paper was originally motivated by the link with the minimum size of a linear set, which we will now introduce.
		
First note that the following lower bound immediately follows from Result \ref{th:graph1}.
	\begin{corollary} A strictly $\F_q$-linear set of rank $h$ in $\PG(1,q^h)$ has at least $q^{h-1}+1$ points.
	\end{corollary}



	Recently, in \cite[Theorem 3.7]{vdv}, the lower bound $q^{h-1}+1$ for linear sets of rank $h$ in $\PG(1,q^h)$ was generalised as follows:
	\begin{result} \label{ondergrens} A linear set of rank $k$ in $\PG(1,q^h)$ containing at least one point of weight one has at least $q^{k-1}+1$ points.
	\end{result}

	\begin{example} Consider again the linear set $S$ from Example \ref{deelrechte}. We see that $S$ is a set of $q^3+1$ points, but when considering $S$ as $L_U$, then it is a linear set of rank $6$ all its points having weight at least $3$. When considering $S$ as $L_V$, where $V=\{(a,b)\mid a\in \F_q,b \in \F_{q^3}\}$, then $L_V$ is a linear set of rank $4$ with one point of weight $3$ (namely $\langle(0,1)\rangle_{q^6}$) and $q^3$ points of weight one. We see that $L_V$ indeed reaches the lower bound for linear sets of rank $4$ containing at least one point of weight $1$. This example also shows that we cannot simply remove the hypothesis that there is a point of weight one in Result \ref{ondergrens}.
	\end{example}
	
		Even though every linear set $L_V$ can be written as a linear set $L_U$ containing at least one point of weight $1$ (by taking a subspace $U$ of $V$ of the correct dimension), the statement of Result \ref{ondergrens} makes clear why we are interested in linear sets without points of weight one.

%

%
	
	We also see that for a rank $k\le h$ linear set in $\PG(1,q^h)$ to have a size lower than $q^{k-1}+1$, all the points in it must have weight at least $2$. Up to our knowledge, the only constructions of $\mathbb{F}_q$-linear sets with all points of weight at least $2$ are obtained by considering linear sets that have geometric field of linearity $\mathbb{F}_{q^s}$ (as done in Example \ref{deelrechte}). 

	It follows from Result \ref{prop22} that for linear sets of rank $k=h$ in $\PG(1,q^h)$ this is the only way to obtain sets with only points of weight at least $2$.

	But we have seen in Example \ref{deelrechte2} that we cannot hope that the converse of Lemma \ref{triv} holds true in general. We will show that the following weaker version holds for a particular class of examples: if all their points have weight more than 1, then they have geometric field of linearity $\mathbb{F}_{q^s}$ for some $s>1$.

More precisely, we will show the following (see Theorem \ref{main}).
\begin{rtheorem}
		If $L$ is a linear set of rank $k$, $4\le k\le h$, in $\PG(1,q^h)$ with one point $P$ of weight $w\ge 2$ and all other points of the same weight $k-w\ge 2$ then $L$ has geometric field of linearity $\mathbb{F}_{q^s}$ with $s\mid w$, $s>1$,  $s\mid h$ and $s\geq k-w$.
		\end{rtheorem}

	Note that we are not claiming that $\F_{q^s}$ is the {\em maximum geometric field of linearity} of the set, see Remark \ref{notmax}.
%
\section{Linear sets as projections}
\subsection{$\Pi$-lines and their type}
	Recall that a rank $k$ linear set $L$ in $\Omega=\PG(r-1,q^h)$ can be viewed as a projection of $\Sigma=\PG(k-1,q)$, a canonical subgeometry of $\Sigma^*=\PG(k-1,q^h)$, from a subspace $\Pi=\PG(k-r-1,q^h)$ disjoint from $\Sigma$ and $\Omega$. Frow now on, the notations $L,\Sigma,\Sigma^*,\Omega$ and $\Pi$ will refer to these unless stated otherwise.

	Furthermore, when we consider a point $P\in\PG(s,q)$ we will assume $P=\langle\v\rangle_{q^h}=\langle(\lambda_1,\lambda_2,\dots,\lambda_{s+1})\rangle_{q^h}$ where $\lambda_1,\dots,\lambda_{s+1}\in \mathbb{F}_q$, i.e. the vector $\v$ representing the point $P$ in the subgeometry $\Sigma$ must have all coordinates in $\F_q$.
	In what follows, all points will be considered as points in $\Sigma^*$ so we will drop the subscript $_{q^h}$ if there is no ambiguity.
	

	The following definitions will be helpful in providing more structure to the linear sets that we are going to explore. 
\begin{definition}
	A line $\ell$ of $\Sigma$ is said to be a $\Pi$\textit{-line} if its extension in $\Sigma^*$ intersects $\Pi$.	A point of $\Pi$ lying on the extension of a $\Pi$-line is a point {\em of rank $2$}.
\end{definition}

\begin{definition} (see also \cite{maartenenik})
Let $Q$ be a point of $\Pi$ of {\em rank $2$} lying on the extension of a $\Pi$-line $\ell$ containing two points $P_1=\langle\v_1\rangle$ and $P_2=\langle\v_2\rangle$, where $P_1,P_2\in \Sigma$, then $Q=\langle\v_1-\alpha \v_2\rangle$ for some $\alpha\in \F_{q^h}\setminus \F_q$. We say that $Q$ has {\em type $S_\alpha$}, where $$S_\alpha:=\left\{\frac{a\alpha+b}{c\alpha+d}\mid a,b,c,d\in\mathbb{F}_q, ad\ne bc\right\}.$$
 We also say that the line $\ell$ is of \textit{type} $S_\alpha$.	
\end{definition}
It follows from \cite[Lemma 2.2]{maartenenik} that the above notion of type is well-defined: for a point $Q$ of rank $2$, if $Q=\langle\v_1-\alpha \v_2\rangle=\langle\v'_1-\alpha' \v'_2\rangle$, then $\alpha'=\dfrac{a\alpha+b}{c\alpha+d}$ for some $a,b,c,d\in\mathbb{F}_q, ad\ne bc$.
Vice versa, if $Q=\langle\v_1-\alpha' \v_2\rangle$ with $\alpha'\in S_\alpha$, then there exist $\v_1',\v_2'$ such that $Q=\langle\v'_1-\alpha\v'_2\rangle$.


Note that the set $S_\alpha$ is simply the orbit of $\alpha$ under the natural action of $\mathrm{PGL}(2,q)$ on the elements of $\mathbb{F}_{q^h}\setminus\mathbb{F}_q$. The following lemma then easily follows.

\begin{lemma}
\label{family} If $[\F_q(\alpha):\F_q]=2$, then $|S_\alpha|=q^2-q$. If $[\F_q(\alpha):\F_q]>2$, then $|S_\alpha|=q^3-q$.
\end{lemma}
\begin{proof} Let $\alpha\in \F_{q^h}$ with $[\F_q(\alpha):\F_q]=2$ and let $\alpha'$ be in $\F_q(\alpha)\setminus \F_q$. Note that, since $\F_q(\alpha)$ is $2$-dimensional over $\F_q$, $|\F_q(\alpha)\setminus \F_q|=q^2-q$ and $\alpha'=a\alpha+b$, for some $a\neq 0 \in \F_q$, $b\in \F_q$. Hence, $\alpha'=\frac{a\alpha+b}{c\alpha+d}$ for $c=0$, $d=1$. Since $ad-bc\neq 0$, $\alpha'\in S_\alpha$, the first statement follows.

Secondly, let $\alpha\in \F_{q^h}$ with $[\F_q(\alpha):\F_q]>2$. Since $|\mathrm{PGL}(2,q)|=q^3-q$, $S_\alpha$ has at most $q^3-q$ elements. It is easy to see that if $\frac{a\alpha+b}{c\alpha+d}=\frac{a'\alpha+b'}{c'\alpha+d'}$ for different elements of $\mathrm{PGL}(2,q)$ defined by $\begin{bmatrix}a&b\\c&d\end{bmatrix},\begin{bmatrix}a'&b'\\c'&d'\end{bmatrix}$, then $\alpha$ satisfies a non-vanishing quadratic equation with coefficients in $\F_q$, a contradiction. Hence, $|S_\alpha|=q^3-q$. 
\end{proof}

	Note that since the sets $S_\alpha$ are orbits of elements in $\mathbb{F}_{q^h}\setminus \mathbb{F}_q$, the different sets $S_\beta$ partition $\mathbb{F}_{q^h}\setminus \mathbb{F}_q$. 
	We also see that if $\alpha'\in S_\alpha$, then $\lambda\alpha'\in S_\alpha$ for all $\lambda\in \F_q^*$. In what follows, we will consider the elements of $S_\alpha$ up scalar multiple in $\F_q^*$.

	\begin{definition} \label{cosets}
	
	Consider the cosets of $\F_q^*$ in the set $S_\alpha$ and let $\bar{S_\alpha}$ denote this set of cosets:	
	$$ \bar{S_\alpha}=\{\{\lambda\alpha'\mid \lambda \in \F_q^*\}\mid \alpha'\in S_\alpha\}.$$
	 	It follows from Lemma \ref{family} that $|\bar{S_\alpha}|=q$ if $[\F_q(\alpha):\F_q]=2$, and $|\bar{S_\alpha}|=q^2+q$ otherwise.
	 \end{definition}

	\subsection{$\Pi$-lines and points of weight at least $2$}

Recall that every $\Pi$-line in $\Sigma$ extends to a line in $\Sigma^*$ containing precisely one point of rank $2$ in $\Pi$. Vice versa, a rank $2$ point lies on the extension of a unique $\Pi$-line: since $\Pi$-lines are contained in the subgeometry $\Sigma$, we see that if (the extensions in $\Sigma^*$ of) two $\Pi$-lines intersect, they do so in $\Sigma$. We conclude that there is a bijection between the set of rank $2$ points in $\Pi$ and the set of $\Pi$-lines in $\Sigma$.

Furthermore, every $\Pi$-line $\ell$ gives rise to a unique point of weight at least $2$ in the linear set $L$: all the points of $\ell$ are projected from $\Pi$ onto the same point, say $P$, of $\Omega$. We also say that $\ell$ {\em corresponds} to the point $P$. More generally, a point of weight $w$ in the linear set that is obtained from the projection of $\Sigma$ corresponds to a $(w-1)$-dimensional subspace $H$ of $\Sigma$; all the lines in $H$ are clearly $\Pi$-lines, but we will see that not all of them have the same type.

We conclude that there is a surjective mapping from the set of $\Pi$-lines to the set of points in $L$ of weight more than $1$. For more details about relation between rank $2$ points and points of weight at least $2$, see e.g. \cite{jon}.

Finally, if $\alpha'\in S_\alpha$, then $\alpha'$ can be expressed as a rational function of $\alpha$, so $\F_q(\alpha)=\F_q(\alpha')$. This ensures that the statement of the folllowing lemma is well-defined.

\begin{lemma}{\label{lemtyp}}

Let $P=\langle\v\rangle\in\Sigma$ be a point lying on a $\Pi$-line $\ell$ of type $S_\alpha$. If $[\mathbb{F}_{q}(\alpha):\mathbb{F}_q]=2$, then $\ell$ is the unique $\Pi$-line through $P$ of type $S_\alpha$.  If $[\mathbb{F}_{q}(\alpha):\mathbb{F}_q]>2$, then there are at most $q+1$ $\Pi$-lines  of type $S_\alpha$  through $P$. Moreover, if there are $q+1$ $\Pi$-lines of type $S_\alpha$ through $P$ then for each $\gamma\in S_\alpha$, there is precisely one point $\langle\u_\gamma\rangle \in \Sigma$ such that $\langle\v-\gamma \u_\gamma\rangle\in\Pi$.

\end{lemma}

\begin{proof} 

We first show that every point different from $P$ on a $\Pi$-line through $P$ gives rise to a unique element of $\bar{S_\alpha}$, that is, to a unique coset of $\F_q^*$ in $S_\alpha$.
Consider a point $Q=\langle \u\rangle$, $Q\neq P$, in $\Sigma$ lying on a $\Pi$-line of type $S_\alpha$ through $P$, then by definition, there is some $\alpha_1\in S_\alpha$ such $R= \langle\v-\alpha_1\u\rangle$ is a point of $\Pi$. Now $\alpha_1\in S_\alpha$ lies in a unique element of $\bar{S_\alpha}$. In order to see that the map sending $Q$ to the coset of $\alpha_1$ is well-defined, note that the only vectors $\u'$ such that $Q=\langle \u'\rangle=\langle \u\rangle$ are of the form $ \u'=\lambda \u$ for some $\lambda\in \F_q^*$. The point $R$ is the unique point of $\Pi$ on the line $PQ$. We see that $R= \langle\v-\alpha_2\u'\rangle$, with $\alpha_2=\frac{1}{\lambda}\alpha_1$ and hence, $\alpha_1$ and $\alpha_2$ indeed define the same element of $\bar{S_\alpha}$. 

We will now show that the above mapping takes different points on $\Pi$-lines of type $S_\alpha$ through $P$ to different elements of $\bar{S_\alpha}$.
To this end, consider two different points, say $\langle \mathbf{u_1}\rangle$ and $\langle \mathbf{u_2}\rangle$ in $\Sigma$ and suppose that $R_1=\langle\v-\alpha'\mathbf{u_1}\rangle$ and $R_2=\langle\v-\lambda\alpha'\mathbf{u_2}\rangle$ are points of $\Pi$. If $R_1=R_2$, then  it easily follows that $\langle \mathbf{u_1}\rangle=\langle \mathbf{u_2}\rangle$, a contradiction since we started with different points in $\Sigma$. If $R_1\neq R_2$, we see that the line $R_1R_2$ intersects $\Sigma$ in the point $\langle\mathbf{u_1}-\lambda \mathbf{u_2}\rangle$, again a contradiction since $\Sigma$ and $\Pi$ are disjoint.

It follows that the number of points on $\Pi$-lines of type $S_\alpha$ through $P$ is at most $|\bar{S_\alpha}|$. Since every line through $P$ contains $q$ points, different from $P$, it follows from Definition \ref{cosets} that the number of $\Pi$-lines is at most 1 if $[\F_q(\alpha):\F_q]=2$ and at most $q+1$ if $[\F_q(\alpha):\F_q]>2$. 

Now assume that there are precisely $q+1$ $\Pi$-lines with type contained in $S_\alpha$ through $P$. Let $\beta_1,\ldots,\beta_{|\bar{S_\alpha}|}$ be a set of representatives of each coset in $\bar{S_\alpha}$. It follows from the above reasoning and the pigeonhole principle that each of the points on $\Pi$-lines through $P$ is of the form $\langle \mathbf{v}-\beta_i\mathbf{u}\rangle$ and moreover, all $\beta_i$ occur when describing these points. Now consider an element $\gamma\in S_\alpha$ and let $\beta_j$ be the coset representative of $\gamma$, then $\gamma=\lambda \beta_j$ for some $\lambda\in \F_q^*$. The statement now follows from the observation that $\langle \mathbf{v}-\gamma\mathbf{u'}\rangle$ with $\mathbf{u'}=\frac{1}{\lambda}\mathbf{u}$ is the same point as $\langle \mathbf{v}-\beta_i\mathbf{u}\rangle$.

\end{proof}	

	
	
	
\begin{lemma}{\label{lemdiv}}
	Let $H$ a $(w-1)$-dimensional subspace of $\Sigma$ corresponding to a point of weight $w\ge3$ in $\Omega=\PG(r-1,q^h)$. If $H$ contains at least $q^{w-1}+q^{w-2}+\dots+1$ $\Pi$-lines of type $S_\alpha$ then the following hold:
	\begin{itemize} \item[(i)] $H$ contains precisely $q^{w-1}+q^{w-2}+\dots+1$ $\Pi$-lines of type $S_\alpha$
	\item[(ii)] 
	
 $[\mathbb{F}_{q}(\alpha):\mathbb{F}_q]=s>2$  and  $s\mid w$.
\end{itemize}
\end{lemma}	

\begin{proof} (i) We count pairs $(P,L)$ where $P$ is a point of $H$ and $L$ is a $\Pi$-line of type $S_\alpha$ through $P$. Let $a$ be the average number of $\Pi$-lines of type $S_\alpha$ through a point of $H$ and $b$ be the number of $\Pi$-lines of type $S_\alpha$. Then 
$$(q+1)b=a|H|,$$ Now $|H|=q^{w-1}+q^{w-2}+\dots+1$ and by Lemma \ref{lemtyp}, $a\leq q+1$, and hence, $b\leq q^{w-1}+q^{w-2}+\dots+1$. It follows that if $b\geq q^{w-1}+q^{w-2}+\dots+1$, then indeed $b=q^{w-1}+q^{w-2}+\dots+1$ and and every point of $H$ lies on exactly $q+1$ $\Pi$-lines of type $S_\alpha$.

(ii)
	 Lemma \ref{lemtyp}  confirms that if every point of $H$ lies on exactly $q+1$ $\Pi$-lines of type $S_\alpha$, we have that $[\mathbb{F}_{q}(\alpha):\mathbb{F}_q]=s>2$. And moreover, that for any point $P=\langle\v\rangle$ in $H$ and any $\beta\in S_\alpha$, there exists a unique point $\langle\u_\beta\rangle\in H$ such that $\langle\v-\beta \u_\beta\rangle\in \Pi$. 
	
	Now let $Q=\langle\v_1\rangle$ be an arbitrary point in $H$. Then, since $\alpha\in S_\alpha$, there exists a unique point $\langle\v_2\rangle\in H$ such that $\langle\v_1-\alpha \v_2\rangle\in \Pi$. Similarly, given $\langle \v_2\rangle$, there exists a unique point $\langle\v_3\rangle\in H$ such that $\langle\v_2-\alpha \v_3\rangle\in\Pi$ and so on. Thus, the point $Q=\langle\v_1\rangle$ gives rise in a unique way to an ordered set of points $\langle \v_2\rangle, \langle \v_3\rangle,\ldots$ such that $\langle \v_{i}-\alpha \v_{i+1}\rangle$ in $\Pi$.
	We will show that the points $\langle \v_1\rangle,\langle \v_2\rangle,\ldots$ obtained in this way are precisely the $\frac{q^{s}-1}{q-1}$ points forming an $(s-1)$-dimensional subspace of $H$.

	First note that the extension of $H$ to $\Sigma^*$ intersects $\Pi$ in a $(w-2)$-dimensional subspace $H'$ of $\Pi$. As the points $\langle\v_1-\alpha \v_2\rangle,\langle\v_2-\alpha \v_3\rangle,\dots$ lie in $H'$, there can be a maximum of $w-1$ such points whose corresponding vectors are linearly independent over $\mathbb{F}$. Suppose that the vectors $\v_1-\alpha \v_2,\v_2-\alpha \v_3,\dots, \v_{t-1}-\alpha \v_{t}$ are independent but that $ \v_t-\alpha \v_{t+1}$ is dependent of the previous vectors. This implies that the points $\langle\v_1-\alpha \v_2\rangle,\langle\v_2-\alpha \v_3\rangle,\ldots, \langle\v_{t-1}-\alpha \v_{t}\rangle$ span a $(t-2)$-dimensional subspace of $H'$ containing the point $\langle \v_t-\alpha \v_{t+1}\rangle$. 
	
	We find that 
	
	\begin{align}
	\v_{t}-\alpha \v_{t+1}=\xi_1 (\v_1-\alpha \v_2)+\dots+\xi_{t-1}(\v_{t-1}-\alpha \v_t)\label{LD}
	\end{align}
	for some $\xi_1,\dots,\xi_{t-1}\in \mathbb{F}_{q^h}$. 
	It follows that $\v_{t+1}$ is an $\F_{q^h}$-linear combination of $\v_1,\v_2,\ldots,\v_{t}$ but since these vectors have all have entries in $\F_q$, we have that $\v_{t+1}=\lambda_1\v_1+\lambda_2\v_2+\dots+\lambda_{{t}}\v_{t}$ for some $\lambda_1,\dots,\lambda_t\in\mathbb{F}_q$.
	Note that the vectors $\v_1,\dots,\v_{t}$ are linearly independent: if they were dependent, the points $\langle\v_1\rangle,\dots,\langle\v_{t}\rangle$ would be contained in a subspace of dimension $t-2$ of $H$, which in turns yields that the points $\v_1-\alpha \v_2,\v_2-\alpha \v_3,\dots, \v_{t-1}-\alpha \v_t$ are contained in a $(t-3)$-dimensional subspace of $H'$, a contradiction.

From \eqref{LD} we find that
	\begin{align*}
		-\alpha\lambda_1&=\xi_1\\
		-\alpha\lambda_2&=\xi_2-\alpha\xi_1\\
		-\alpha\lambda_3&=\xi_3-\alpha\xi_2\\
		\vdots\\
		-\alpha\lambda_{t-1}&=\xi_{t-1}-\alpha\xi_{t-2}\\
		1-\alpha\lambda_t&=-\alpha\xi_{t-1}.
	\end{align*}
	Eliminating $\xi_1,\ldots,\xi_{t-1}$ from the above system of equations yields that $\alpha$ must satisfy the equation
	$$\alpha^{t}\lambda_1+\alpha^{t-1}\lambda_2+\dots+\alpha \lambda_{t}=1.$$
	As $\mathbb{F}_q(\alpha)=\mathbb{F}_{q^s}$ is a subfield of $\mathbb{F}_{q^h}$, and as $\alpha$ satisfies a degree $t$ polynomial, we must have $ s\le t\le w$ and $s\mid h$.

	We now show that $s=t$. In order to do so, let the minimal polynomial of $\alpha\in \mathbb{F}_{q^h}$ be \begin{align}\alpha^{s}\nu_1+\alpha^{s-1}\nu_2+\dots+\alpha \nu_{s}=1.\label{min2}\end{align}
	Consider a point $Q=\langle \w_1\rangle$ of $H$, then we have seen before that there are unique points $\langle \w_2\rangle,\ldots,\langle \w_s\rangle$  such that $\langle \w_1-\alpha \w_2\rangle, \langle \w_2-\alpha \w_3\rangle,\dots, \langle \w_{s-1}-\alpha \w_{s}\rangle$ are points of $H'$. Furthermore, the above argument shows that $\langle \w_1\rangle,\ldots,\langle \w_s\rangle$ span an $(s-1)$-dimensional subspace of $H$.
	Denote this $(s-1)$-space by $H_Q$.

	Note that
	\begin{align*}&\langle \w_1-\alpha\w_2+\alpha(\w_2-\alpha \w_3)+\alpha^2(\w_3-\alpha \w_4)+\ldots+\alpha^{s-2}(\w_{s-1}-\alpha \w_s)\rangle\\&=\langle \w_1-\alpha^{s-1} \w_s\rangle\end{align*} is a point of $H'$. Similarly, it easily follows that the points $\langle \w_2-\alpha^{s-2} \w_s\rangle,\dots, \langle \w_{s-1}-\alpha \w_{s}\rangle$ also lie in $H'$. This implies that the point
	
	\begin{align*}
		&\langle \w_{s}-\alpha({\nu_1\w_1+\nu_2\w_2+\dots+\nu_{s}\w_{s}})\rangle\\
		=&\langle(\alpha^{s}\nu_1+\alpha^{s-1}\nu_2+\dots+\alpha \nu_{s})\w_{s}-\alpha({\nu_1\w_1+\nu_2\w_2+\dots+\nu_{s}\w_{s}})\rangle\\
		=&\langle \alpha\nu_1(\alpha^{s-1}\w_{s}-\w_1)+\alpha\nu_2(\alpha^{s-2}\w_s-\w_2)+\dots+\alpha\nu_{s-1}(\alpha \w_s-\w_{s-1})\rangle
	\end{align*}
	is also contained in $H'$, where we have used \eqref{min2} in the first equality. Since $$\langle \w_{s}-\alpha({\nu_1\w_1+\nu_2\w_2+\dots+\nu_{s}\w_{s}})\rangle\in H'$$ and the point $ \langle \w_{s+1}\rangle$ such that $\langle \w_s-\alpha \w_{s+1}\rangle\in \Pi$ is unique, we obtain that $$\langle \w_{s+1}\rangle=\langle{\nu_1\w_1+\nu_2\w_2+\dots+\nu_{s}\w_{s}}\rangle,$$ which is clearly contained in $H_Q$.

	Recall that $ \w_1, \w_2,\ldots, \w_s$ forms a basis for the $s$-dimensional $\F_q$-vector space determining $H_Q$. Hence, there is a unique $\F_q$-linear map $\phi$ which satisfies $\phi(\w_i)=\w_{i+1}$ for $i=1,\ldots,s-1$ and $\phi(\w_{s})=\nu_1\w_1+\nu_2\w_2+\dots+\nu_{s}\w_{s}$. With respect to the basis $\w_1,\ldots,\w_s$, this map is represented by the matrix 
	
	\begin{align*}
		A=
		\begin{bmatrix}
			0&0&\cdots&0&\nu_1\\
			1&0&\cdots&0&\nu_2\\
			0&1&\cdots&0&\nu_3\\
			\vdots&\vdots&\ddots&\vdots&\vdots\\
			0&0&\cdots&1&\nu_{s}
		\end{bmatrix}.
	\end{align*}

From the definition of $\phi$, we see that for $1\leq i\leq s$, $\langle \w_i-\phi(\w_i)\rangle$ is a point of $\Pi$. But since $\phi$ is an $\F_q$-linear map, we have that for a point $\w=\mu_1\w_1+\dots+\mu_{s}\w_{s}$, with $\mu_i\in \F_q$, 

\begin{align*}\langle \w-\phi(\w)\rangle&=
 \langle \mu_1\w_1+\dots+\mu_{s}\w_{s}-\phi(\mu_1\w_1+\dots+\mu_{s}\w_{s})\rangle\\
  &=\langle \mu_1\w_1+\dots+\mu_{s}\w_{s}-\mu_1\phi(\w_1)-\dots-\mu_{s}\phi(\w_{s}))\rangle\\
   &= \langle \mu_1(\w_1-\phi(\w_1))+\ldots+\mu_{s}(\w_{s}-\phi(\w_{s}))\rangle
 \end{align*}
 is a point of $\Pi$ too.
	In other words, for each point $\langle \w\rangle\in H_Q$, $\langle \phi(\w)\rangle$ is the unique point $\w'$ such that $\langle \w-\alpha\w'\rangle$ in $\Pi$. Since $\phi(H_Q)=H_Q$ we are done.
	(In fact, $A$ is the companion matrix of the minimal polynomial for $1/\alpha$ and the map $\phi$ generates a Singer cycle on the points of $H_Q$).
	
	 It is clear from the construction that if for two points $Q,R\in H$, the spaces $H_Q$ and $H_R$ have a point in common, then $H_Q$ and $H_R$ coincide. Moreover, since every point $S$ lies determines a subspace $H_S$, we find that $H$ is partitioned by the $(s-1)$-spaces $H_P$,  $P\in H$. The existence of this $(s-1)$-spread	 in the $(w-1)$-space $H$ shows that $s\mid w$.

\end{proof}

\begin{remark}\label{srem} It follows from the construction of Desarguesian spreads from Singer cycles that the $(s-1)$-spread of $H$, constructed in Lemma \ref{lemdiv}, is a Desarguesian spread (see e.g. \cite{dye,spreads}). The fact that the spread is normal (which also implies Desarguesian if $w>2s$) is easy to see geometrically: 	the extension of two different $(s-1)$-spread elements $H_P$ and $H_Q$ of $H$ intersect $\Pi$ in two disjoint $(s-2)$ subspaces, say $H_P'$ and $H_Q'$. As for any point $R \in span(H_P,H_Q)$ the space $H_R'$ in $\Pi$ is a subspace of $span(H_P',H_Q')$, the spread element $H_R$ must be in $span(H_P,H_Q)$. 

%
\end{remark}

	\subsection{Constructing $\Pi$-lines}
	Lemma \ref{lemdiv} tells us what happens when there are many $\Pi$-lines of the same type. The following lemma provides a way to find such lines. The first part of this lemma was already proven  for linear sets in $\PG(1,q^5)$ in \cite{maartenenik}.

\begin{lemma}{\label{lemreg}}
	If there are three disjoint $\Pi$-lines $\ell_1,\ell_2$ and $\ell_3$ in a $3$-dimensional subspace $H$ of $\Sigma$, whose extensions in $\Sigma^*$ are contained in a line $M$ of $\Pi$, then all the lines in the unique regulus $\mathcal{R}$ defined by $\ell_1,\ell_2,\ell_3$ are $\Pi$-lines of the same type.
	
	Moreover if there exists one more line $\ell$ in $H$, disjoint from all the lines of $\mathcal{R}$, whose extension in $\Sigma^*$ intersects $M$, then $H$ is partitioned by $\Pi$-lines of the same type $S_\alpha$ with $[\mathbb{F}_q(\alpha):\mathbb{F}_q]=2$.  
	
	Vice versa, if a $3$-space $H$ contains $3$ disjoint $\Pi$-lines $\ell_1$, $\ell_2$, $\ell_3$ of the same type $S_\alpha$ whose extensions in $\Sigma^*$ are contained in a line $M$ of $\Pi$, and $[\mathbb{F}_q(\alpha):\mathbb{F}_q]=2$, then the space $H$ is partitioned by $\Pi$-lines of type $S_\alpha$. 
\end{lemma}

\begin{proof} 	
	Let $\ell_1,\dots,\ell_{q+1}$ be the $q+1$ lines of the regulus $\mathcal{R}$ in $H$ defined by $\ell_1,\ell_2,\ell_3$. Assume that the extensions of $\ell_1,\ell_2,\ell_3$, say $\bar{\ell_1},\bar{\ell_2},\bar{\ell_3}$, meet the line $M$ of $\Pi$. 
	Let $\mathcal{R}'$ be the $q^h+1$ lines of the regulus defined by $\bar{\ell_1},\bar{\ell_2},\bar{\ell_3}$ and let $m_1,m_2,m_3$ be transversals to $\mathcal{R}$. We then see that $\bar{m_1},\bar{m_2},\bar{m_3}$ are transversals to $\mathcal{R}'$ in $\Sigma$. Furthermore, $M$ intersects $\bar{\ell_1},\bar{\ell_2},\bar{\ell_3}$ so $M$ is a transversal to $\mathcal{R}'$ too. Since $\ell_i$, $1\leq i\leq q+1$ intersects $m_1,m_2,m_3$ we have that $\bar{\ell_i}$ intersects $\bar{m_1},\bar{m_2},\bar{m_3}$,  so $\bar{\ell_i}$ is an element of $\mathcal{R}'$. Since $M$ is a transversal to $\mathcal{R}'$, we conclude that all lines $\bar{\ell_i}$ intersect the line $M$.
	
	Now consider the points $\ell_i\cap m_1=P_i,\ell_i\cap m_2=Q_i$ and $\bar{\ell_i}\cap M=S_i$, $i=1,\ldots,q+1$.

Let $P_1=\langle \u_1\rangle$, $Q_1=\langle \v_1\rangle$, $P_2=\langle \u_2\rangle$, $Q_2=\langle \v_2\rangle$. Note that $\u_1,\u_2,\u_3,\u_4$ are linearly independent over $\F_{q^h}$. Without loss of generality (since $\mathrm{PGL}(4,q)$ acts transitively on the frames of $H$), we may assume that the point $\langle \u_1+\u_2+\v_1+\v_2\rangle$ is a point of $\ell_3$. It then follows that the points $P_i$, $i=3,\ldots,q+1$ are of the form $\langle \u_1+\lambda_i\u_2\rangle$ and $Q_i$ of the form $\langle \v_1+\lambda_i\v_2\rangle$ where $\lambda_i \in \F_q^*$.
We have that $S_1=\langle \u_1-\alpha\v\rangle$ for some $\alpha\in \F_{q^h}\setminus \F_q$ and $S_2=\langle \u_1-\alpha'\v_2\rangle$ for some $\alpha'\in \F_{q^h}\setminus \F_q$.  Expressing that $\bar{\ell_i}$ intersects $m_1,m_2$ and $M$ yields that there are $\xi_1,\xi_2,\psi\in \F_{q^h}$ such that
$$\xi_1(\u_1-\alpha\v_1)+\xi_2(\u_2-\alpha'\v_2)=\u_1+\lambda_i \u_2-\psi_i(\v_1+\lambda_i\v_2).$$

Since $\u_1,\u_2,\u_3,\u_4$ are linearly independent over $\F_{q^h}$, it follows that
$\xi_1=1$, $\xi_2=\lambda_i$, $\psi_i=\alpha$ and $\alpha'=\alpha$. It follows that $\psi_i=\alpha'=\alpha$. So all lines $\ell_i$ are of the same type $S_{\alpha}$.

%
%


Now assume that there is an additional line $\ell$ in $\Sigma$, not contained in $\mathcal{R}$, whose extension $\bar{\ell}$ contains a point of $M$; it follows from the first part that $\ell$ is a $\Pi$-line of type $S_\alpha$. Thus there exist points ${\langle\u\rangle}$ and $\langle\v\rangle$ in $\ell$ with $\u=\lambda_1\u_1+\lambda_2\v_1+\lambda_3\u_2+\lambda_4\v_2$, $\v=\mu_1\u_1+\mu_2\v_1+\mu_3\u_2+\mu_4\v_2$ such that $\langle\u-\beta\v\rangle\in M$ i.e. 	\begin{equation}\label{eq1}
		\u-\alpha \v=\xi_1( \u_1-\alpha \v_1)+\xi_2 (\u_2-\alpha \v_2).
	\end{equation}
	Solving (\ref{eq1}) by equating the coefficients of $\u_1,\u_2,\v_1$ and $\v_2$, we get
	\begin{equation}\label{eq2}
		\mu_1\alpha^2+(\mu_2-\lambda_1)\alpha-\lambda_2=0
	\end{equation}
	and
	\begin{equation}\label{eq3}
		\mu_3\alpha^2+(\mu_4-\lambda_3)\alpha-\lambda_4=0.
	\end{equation}
	As $\alpha\notin\mathbb{F}_q$, $\alpha$ satisfies a degree 2 equation given by (\ref{eq2}) or (\ref{eq3}). Thus $[\mathbb{F}_q(\alpha):\mathbb{F}_q]=2$. 
	
	Now let the minimal polynomial of $\alpha$ over $\mathbb{F}_q$ be 
	\begin{equation}\label{eq4}
		a\alpha^2+b\alpha-1=0.
	\end{equation}
	As the equations (\ref{eq2}), (\ref{eq3}) and (\ref{eq4}) are multiples of each other, we must have $\mu_1=a\lambda_2,\mu_2=b\lambda_2+\lambda_1,\mu_3=a\lambda_4$ and $\mu_4=b\lambda_4+\lambda_3$. 
	
	We now see that a point  $\langle\u'\rangle$ with $\u'=\nu_1\u_1+\nu_2\v_1+\nu_3\u_2+\nu_4\v_2$ in $H$ lies on a line of type $S_\alpha$: consider the point $\langle\v'\rangle$ with $\v'=a\nu_2\u_1+(b\nu_2+\nu_1)\v_1+a\nu_4\u_2+(b\nu_4+\nu_3)\v_2$, then

	\begin{align*}
		\langle\u'-\alpha\v'\rangle=&\langle\nu_1(\u_1-\alpha\v_1)+\nu_2((1-b\alpha)\v_1-a\alpha\u_1)\\
		&+\nu_3(\u_2-\alpha\v_2)+\nu_4((1-b\alpha)\v_2-a\alpha\u_2)\rangle\\
		=&\langle\nu_1(\u_1-\alpha\v_1)+\nu_2(a\alpha^2\v_1-a\alpha\u_1)\\
		&+\nu_3(\u_2-\alpha\v_2)+\nu_4(a\alpha^2\v_2-a\alpha\u_2)\rangle\\
		=&\langle(\nu_1-a\alpha\nu_2)(\u_1-\alpha\v_1)+(\nu_3-a\alpha\nu_4)(\u_2-\alpha\v_2)\rangle\in M.
	\end{align*}
	Hence every point in $H$ is contained in a $\Pi$-line of type $S_\alpha$. By Lemma \ref{lemtyp} every point can lie on at most one $\Pi$-line of type $S_\alpha$. Therefore $H$ must be partitioned by $\Pi$-lines of type $S_\alpha$.


Vice versa, suppose now that the $3$-space $H$ contains $3$ disjoint $\Pi$-lines $\ell_1$, $\ell_2$, $\ell_3$ of the same type $S_\alpha$ with $[\mathbb{F}_q(\alpha):\mathbb{F}_q]=2$ such that their extension meets $\Pi$ in $M$. It follows from the first part of the proof that we can take points $\langle \u_1\rangle$, $\langle \v_1\rangle$ on $\ell_1$ and $\langle \u_2\rangle,\langle \v_2\rangle$ on $\ell_2$ such that the line $M$ contains the points $S_1=\langle \u_1-\alpha\v\rangle$ and $S_2=\langle \u_1-\alpha\v_2\rangle$. Writing the minimal polynomial for $\alpha$ as $a\alpha^2+b\alpha-1=0$, we find again (as above) that for any point $\langle\u'\rangle$ with $\u'=\nu_1\u_1+\nu_2\v_1+\nu_3\u_2+\nu_4\v_2$ we have the point $\langle\v'\rangle$ with $\v'=a\nu_2\u_1+(b\nu_2+\nu_1)\v_1+a\nu_4\u_2+(b\nu_4+\nu_3)\v_2$ satisfies $\langle\u'-\alpha\v'\rangle\in M$. And hence, the space $H$ is indeed partitioned by $\Pi$-lines of type $S_\alpha$. 
\end{proof}

\section{Linear sets on a line with all points of weight at least two}
\subsection{Finding a subfield}
	We now turn our attention towards linear sets contained in a line, i.e. $\Omega=\PG(1,q^h)$ when the linear set $L$ of rank $k,4\le k\le h$, in $\Omega$ has one point of weight $k-2$ and all others of weight 2.  In Lemma \ref{thr} and Theorem \ref{thrm} we prove the existence of a proper subfield of $\mathbb{F}_{q^h}$ and later, we will prove that $L$ has this field as geometric field of linearity.

\begin{lemma}{\label{thr}}
	If $L$ is a rank 4 linear set in $\PG(1,q^h)$, $h\ge4$, with all points of weight 2, then $\mathbb{F}_{q^h}$ contains the subfield $\mathbb{F}_{q^2}$ (i.e. $h$ is even).
\end{lemma}

\begin{proof}
	A rank $4$ linear set $L$ in $\Omega=\PG(1,q^h) $ can be viewed as a projection of $\Sigma=\PG(3,q)$, a canonical subgeometry of $\Sigma^*=\PG(3,q^h)$, from $\Pi=\PG(1,q^h) $ onto $\Omega$, where $\Pi$ and $\Omega$ are in $\Sigma^*$ with $\Pi$ disjoint from $\Sigma$ and $\Omega$. In this case all the $q^2+1$ weight 2 points in $L$ correspond to different $\Pi$-lines of rank 2 each meeting the line $\Pi$ in a point. 
	 Using Lemma \ref{lemreg} we see that all $\Pi$-lines are of type $S_\alpha$ with $[\F_q(\alpha):\F_q]=2$.

\end{proof}

	\begin{remark} It will follow from our more general approach that the linear sets considered in Lemma \ref{thr} have geometric field of linearity $\F_{q^2}$ (see Main Theorem). In this particular case however, it is not too hard to see that the set $L$ is in fact $\F_{q^2}$-linear and isomorphic to $\PG(1,q^2)$. A different way to deduce this fact is to use a characterisation of sublines (e.g. given in \cite[Theorem 1.5]{sara}): a set of $q^2+1$ points in $\PG(1,q^h)$ is a subline $\PG(1,q^2)$ if and only if it is closed under taking $\F_q$-sublines determined by any three points of the set.
	
	\end{remark}


\begin{theorem}\label{thrm}
	If $L$ is a linear set of rank $k,5\le k\le h$, in $\PG(1,q^h)$ with one point of weight $k-2$ and all other points of weight 2 then $\mathbb{F}_{q^h}$ must contain a proper subfield $\mathbb{F}_{q^s}$ with $s\mid k-2$, $s>1$.
\end{theorem}

\begin{proof}
	Recall that $L$ can be viewed as a projection of $\Sigma=\PG(k-1,q)$, a canonical subgeometry of $\Sigma^*=\PG(k-1,q^h)$, from $\Pi=\PG(k-3,q^h)$ onto $\Omega=\PG(1,q^h)$, where $\Pi$ and $\Omega$ are subspaces of $\Sigma^*$, with $\Pi$ disjoint from $\Sigma$ and $\Omega$. We have seen that each of the weight 2 points determines a unique $\Pi$-line and the weight $k-2$ point corresponds to a $(k-3)$-dimensional subspace $H$ of $\Sigma$. Denote the set of $\Pi$-lines corresponding to the weight 2 points by $\mathcal{S}$. Note that the $\Pi$-lines in $\mathcal{S}$ partition the set of points in $\Sigma\setminus H$. Also note that the extension of $H$ to $\Sigma^*$ intersects $\Pi$ in a hyperplane $H'$ of $\Pi$.
	
	Let $\ell_1$ and $\ell_2$ be different lines from $\mathcal{S}$ and let $J=span(\ell_1,\ell_2)$ be the $3$-dimensional subspace of $\Sigma$ spanned by $\ell_1$ and $\ell_2$. Since $\ell_1$ and $\ell_2$ are $\Pi$-lines, the extension $\bar{J}=span(\ell_1,\ell_2)$ of $J$ in $\Sigma^*$ contains the points $\bar{\ell_1}\cap \Pi$ and $\bar{\ell_2}\cap \Pi$. If $\dim(\bar{J}\cap\Pi)>1$, then all the points in $J$ would be projected onto a single point in $\Omega$, a contradiction since $\ell_1$ and $\ell_2$ correspond to different points of $L$. We conclude that $\bar{J}\cap\Pi$ is a line $M$. Since $\dim(\Sigma)\geq 5$, $\dim(J)=3$ and $span(J,H)=\Sigma$, $J\cap H$ is a line $\ell$. The lines $\ell_1,\ell_2$ and $\ell$ define a unique regulus $\mathcal{R}$ in $J$. Note that $\ell$ is contained in the $(k-3)$-space $H$, and hence, its extension $\bar{\ell}$ meets the $(k-4)$-space $H'$ in a point which then necessarily lies on $M$.
	
	
	 By Lemma \ref{lemreg} all the $\Pi$-lines in $\mathcal{R}$ are of the same type, say $S_\alpha$. Repeating the same argument for any pair $\ell_1,\ell_2'$ of lines in $\mathcal{S}$, we see that all the $\Pi$-lines in $\mathcal{S}$ are of type $S_\alpha$. 
	
Now fix a $\Pi$-line $\ell\in S$ of type $S_\alpha$. Consider the set of lines $\mathcal{T}=\{span(\ell,\ell')\cap H\mid \ell'\in \mathcal{S}\setminus\{\ell\}\}$. First suppose that each of the spaces $span(\ell,\ell')$, $\ell'\in \mathcal{S}\setminus\{\ell\}$, contains at most $q-1$ lines of $\mathcal{S}\setminus\{\ell\}$, then $\mathcal{T}$ is a set of at least $q^{k-3}+q^{k-4}+\dots+1$ different $\Pi$-lines of type $S_\alpha$ in $H$. By Lemma \ref{lemdiv} $\mathbb{F}_{q^h}$ must have a subfield $\mathbb{F}_{q^s}$ with $s>2,s\mid k-2$ and we are done.

So suppose that there is a line $\ell'\in \mathcal{S}\setminus\{\ell\}$ such that $span(\ell,\ell')$ contains more than $q-1$ lines of $\mathcal{S}\setminus\{\ell\}$. Then by Lemma \ref{lemreg} we will have $[\F_q(\alpha):\F_q]=2$, and hence, $\F_{q^2}$ is a subfield of $\F_{q^h}$. We also find that $span( \ell,\ell')\setminus (H\cap span( \ell,\ell'))$ is partitioned by lines of $\mathcal{S}$. Furthermore, again invoking Lemma \ref{lemreg}, since $[\F_q(\alpha):\F_q]=2$, we find that each $3$-space spanned by $\ell$ and a line $\ell''$ of $\mathcal{S}\setminus \{\ell\}$ is partitioned into $q^2+1$ $\Pi$-lines of type $S_\alpha$, one of which is $\ell$ and one of which is the intersection of $span(\ell,\ell'')$ with $H$. Hence, we find that $q^2-1$ needs to divide $q^{k-2}-1$, which in turn implies that $2\mid k-2$.

%
\end{proof}

\begin{remark}\label{2rem}
	If $s=2$ in Theorem \ref{thrm} and $\ell$ is of type $S_\alpha$, then $[\mathbb{F}_q(\alpha):\mathbb{F}_q]=2$ and all the $\Pi$-lines in $\mathcal{T}$ are of type $S_\alpha$. By Lemma \ref{lemtyp}, every point in $H$ lies in at most one $\Pi$-line of type $S_\alpha$, and it is easy follows that each of the points in $H$ lies on exactly one $\Pi$-line of type $S_\alpha$.
	The line spread determined by $\Pi$-lines of type $S_\alpha$ of $H$ is a normal spread. This follows easily from the construction: for $m_1,m_2\in \mathcal{T}$, $\ell\in \mathcal{S}$, the subspace $span(\ell,m_1,m_2)$ is partitioned by elements of $\mathcal{S}\cup \mathcal{T}$.

	
\end{remark}

\subsection{(Re)constructing linear sets}
	In the following construction we use the projection point of view to construct an $\F_q$-linear set which will later prove to be the unique way of describing linear sets satisfying the hypotheses used in Lemma \ref{thr} and Theorem \ref{thrm}. We show that these linear sets have geometric field of linearity $\F_{q^s}$.

	We use the standard vectors $\{\e_i\mid i=1,\dots,k\}$, where $\e_i=(0,\dots,0,1,0,\dots,0)$ is a vector of length $k$ with $1$ in the $i_\text{th}$ position and $0$ elsewhere, as the basis of $V$, the $k$-dimensional vector space over $\mathbb{F}_{q}$ defining $\Sigma=\PG(V)=\PG(k-1,q)$. We use the same set of vectors as the basis of $V'$, a $k$-dimensional vector space over $\mathbb{F}_{q^h}$ defining $\Sigma^*=\PG(V')=\PG(k-1,q^h)$, thus $\Sigma$ is a canonical subgeometry of $\Sigma^*$.
	
	
 
\begin{construction}{\label{const}}
	Let $\Sigma^*=\PG(k-1,q^h)=span_{q^h}( \langle\e_1\rangle,\dots, \langle\e_k\rangle),4\le k\le h,$. Let $\Sigma$ be the canonical subgeometry embedded in $\Sigma^*$, given by all points $\langle u\rangle_{q^h}$, where $u\in span_q( \e_1,\ldots,\e_k)$. 
Suppose that $k=rs+2$ and $s\mid h$.
	Partition the ordered set $\{\e_1,\dots,\e_{k}\}=\{\e_1,\dots,\e_{rs},\e_{rs+1},\e_{rs+2}\}$ into $r+1$ parts, $r$ of which have size $s\ge2$ and are called $A_1,\ldots, A_r$, and are given by the ordered sets
	\begin{align*}
		A_i=&\{\e_{(i-1)s+1},\dots,\e_{is}\}=\{\e_{i,1},\dots,\e_{i,s}\}
	\end{align*}
	and the last part is of size $2$ given by the set $$B=\{\e_{rs+1},\e_{rs+2}\}.$$

	Consider $\alpha \in{\mathbb{F}}_{q^h}\backslash\mathbb{F}_q$ generating a degree $s$ extension of $\mathbb{F}_q$ (i.e. $[\mathbb{F}_q(\alpha):\mathbb{F}_q]=s$). With each $A_i, i=1,\dots,r$, we associate the $(s-2)$-dimensional subspace $\Pi_i$ of $\Sigma^*$ given by $$\Pi_i=span(\langle \e_{i,1}-\alpha \e_{i,2}\rangle,\langle \e_{i,2}-\alpha \e_{i,3}\rangle,\dots,\langle \e_{i,s-1}-\alpha \e_{i,s}\rangle).$$
	Let $\Pi=\PG(k-3,q^h)$ be the subspace $$\Pi:=span( \Pi_1,\Pi_2,\dots,\Pi_r,\langle \e_{1,s}-\beta_1\e_{2,s}\rangle,\langle \e_{2,s}-\beta_2\e_{3,s}\rangle,\dots,\langle \e_{r-1,s}-\beta_{r-1}\e_{r,s}\rangle,\langle \e_{rs+1}-\alpha\e_{rs+2}\rangle)$$ with $\beta_1,\dots,\beta_{r-1}\in\mathbb{F}_{q^h}\backslash\mathbb{F}_{q^s}$ such that $\Pi$ is a $(k-3)$-space disjoint from $\Sigma$ and $\Omega$.


	Finally, let $\Omega=\PG(1,q^h)$ be the subspace $$\Omega:=span(\langle \e_{r,s}\rangle,\langle\e_{rs+2}\rangle) .$$
\end{construction}

The projection of $\Sigma$ from $\Pi$ onto $\Omega$ in Construction \ref{const} is an $\mathbb{F}_q$-linear set, and we will show now that it has geometric field of linearity $\F_{q^s}$. 

\begin{lemma}{\label{linearity}}
	The linear set $L$ obtained from Construction \ref{const} has geometric field of linearity $\mathbb{F}_{q^s}$. 
\end{lemma}

\begin{proof}
	By the construction any point $\langle (\lambda_1,\dots,\lambda_{k})\rangle_{q^h}=\langle (\lambda_{1,1},\dots,\lambda_{r,s},\lambda_{rs+1},\lambda_{rs+2})\rangle_{q^h}$ in $\Sigma$ is projected onto the point $\langle(0,\dots,0,\chi_1,0,\chi_2)\rangle_{q^h}$ in $\Omega$ where
		\begin{align*}
			\chi_1= \beta_{1}\beta_{2}\dots\beta_{r-1}&(\alpha^{s-1}\lambda_{1,1}+\alpha^{s-2}\lambda_{1,2}+\dots+\lambda_{1,s})\\
			+\beta_2\beta_3\dots\beta_{r-1}&(\alpha^{s-1}\lambda_{2,1}+\alpha^{s-2}\lambda_{2,2}+\dots+\lambda_{2,s})\\
			+\ldots\quad\quad&\\
			+\beta_{r-1}&(\alpha^{s-1}\lambda_{r-1,1}+\alpha^{s-2}\lambda_{r-1,2}+\dots+\lambda_{r-1,s})\\
			+&(\alpha^{s-1}\lambda_{r,1}+\alpha^{s-2}\lambda_{r,2}+\dots+\lambda_{r,s})
		\end{align*}
	and
	\begin{align*}
		\chi_2= \lambda_{rs+2}+\alpha \lambda_{rs+1}.
	\end{align*}

	Thus the linear set $L$ consists of all points $\langle\chi_1\e_{r,s}+\chi_2\e_{rs+2}\rangle_{q^h}$ with
	$$(\chi_1,\chi_2)=(\beta_1\beta_2\dots\beta_{r-1}f_1(\alpha)+\beta_2\beta_3\dots\beta_{r-1}f_2(\alpha)+\dots+f_r(\alpha),\lambda_{rs+2}+\alpha \lambda_{rs+1}),$$
where $f_i(\alpha)$ is an arbitrary polynomial with degree at most $s-1$ and coefficients in $\mathbb{F}_q$ and $\lambda_{rs+1}, \lambda_{rs+2}\in \F_q$. Recall that any element of $\mathbb{F}_{q^s}$ can be represented as a polynomial in $\alpha$ with coefficients in $\mathbb{F}_q$ and degree at most $s-1$ and vice versa, each such polynomial defines an elements of $\mathbb{F}_{q^s}$. It follows that 	$$(\chi_1,\chi_2)=(\beta_1\beta_2\dots\beta_{r-1}\gamma_1+\beta_2\beta_3\dots\beta_{r-1}\gamma_2+\dots+\gamma_r,\lambda_{rs+2}+\alpha \lambda_{rs+1})$$
where $\gamma_i$, $1\leq i\leq r$ an arbitrary element of $\F_{q^s} $ and $\lambda_{rs+1}, \lambda_{rs+2}\in \F_q$. 
Now consider the set of points $M$ of the form $\langle\chi_1'\e_{r,s}+\chi_2'\e_{rs+2}\rangle_{q^h}$
where $$(\chi_1',\chi_2')=(\beta_1\beta_2\dots\beta_{r-1}\gamma_1+\beta_2\beta_3\dots\beta_{r-1}\gamma_2+\dots+\gamma_r,\gamma_{r+1}),$$ where $\gamma_i\in \F_{q^s}$, $1\leq i\leq r+1$.
Since the set of vectors of the form $(\chi_1',\chi_2')$ is closed under addition and $\F_{q^s}$-multiplication, the set $M$ is an $\F_{q^s}$-linear set of $\F_{q^s}$-rank $r+1$ (and an $\F_q$-linear set of rank $rs+s$). Note that, as an $\F_{q^s}$-linear set, $M$ contains points of  $\F_{q^s}$-weight $1$, and hence, $M$ is not an $\F_{q^{s^i}}$-linear set for any $i>1$.
It is clear that every point in $L$ also belongs to $M$, so $L\subseteq M$. Moreover, if $P$ is a point of $M$, then either $P=\langle \e_{r,s}\rangle_{q^h}$ which also belongs to $L$, or $P=\langle (\beta_1\beta_2\dots\beta_{r-1}\gamma'_1+\beta_2\beta_3\dots\beta_{r-1}\gamma'_2+\dots+\gamma'_r)\e_{r,s}+1\e_{rs+2}\rangle_{q^h}$, where $\gamma_i'=\gamma_i/\gamma_{r+1}$, which then clearly belongs to $L$ too. We conclude that $L=M$, and since $M$ is an $\F_{q^s}$-linear set, $L$ has geometric field of linearity $\F_{q^s}$.
%
%
%
%
%
\end{proof}

\begin{remark}\label{club}
	The linear set $L$ in Lemma \ref{linearity} contains $q^{rs}+1$ points: one corresponds to $\langle\e_{r,s}\rangle$ and the other $q^{rs}$ correspond to the different choices of the polynomials $f_1,\dots,f_r$ (or equivalently, the choice of elements $\gamma_1,\ldots,\gamma_r$). The $\F_q$-weight of the point $\langle\e_{r,s}\rangle$ is $ls=k-2$ and every point other point in $L$ is of weight $2$. We have shown that $L=M$, where $M$ is an $\mathbb{F}_{q^s}$-linear set, forming an $l$-club: it has $(q^s)^r$ points of $\F_{q^s}$-weight $1$ and $1$ point of $\F_{q^s}$-weight $l$.
We also see that if $l=1$ and $r=1$, then $L$ is a set of $q^{s}+1$ points (with $s=k-2$) which is the set of points of an $\F_{q^s}$-subline of $\PG(1,q^h)$.
\end{remark}

\begin{remark} Recall from the introduction that every $\F_q$-linear set can be described as the set of spread elements of the Desarguesian spread $\mathcal{D}$ in $\PG(2h-1,q)$ intersecting some fixed subspace. On one hand, the $\F_q$-linear set $L$ of Construction \ref{const} can be seen as those elements of $\mathcal{D}$ meeting a fixed $(k-1)$-space $\pi$, where there is one element, say $\sigma_P$, of $\mathcal{D}$ meeting $\pi$ in a $(k-3)$-space $\mu$ and the other elements of $\mathcal{D}$ meet $\pi$ in a line or are disjoint from $\pi$. 
On the other hand, we have shown in Lemma \ref{linearity} that $L$ can be seen as the set of elements of the Desarguesian spread $\mathcal{D}$ which meet a fixed $(k+s-3)$-space $\pi'$, where $\pi'$ contains $\pi$ and such that $\sigma_P\cap \pi=\sigma_P\cap \pi'$. Hence, there is one element, $\sigma_P$ meeting $\pi'$ in a $(k-3)$-space and all other elements of $\mathcal{D}$ either meet $\pi'$ in an $(s-1)$-dimensional space or are disjoint from $\pi'$.
The set of $(s-1)$-dimensional subspaces of $\pi'$ obtained as the intersections of elements of $\mathcal{D}$ with $\pi'$ are contained in the unique Desarguesian $(s-1)$-subspread of $\mathcal{D}$ (see also \cite{spreads}).

\end{remark}

\subsection{The proof of the main theorem}
The following proposition forms the base case for our main theorem.
\begin{proposition}\label{coromain} If $L$ is a rank 4 linear set in $\PG(1,q^h)$, $h\ge4$, with all points of weight 2, then $L$ is an $\F_{q^2}$-linear set $\cong\PG(1,q^2)$.
If $L$ is a linear set of rank $k,5\le k\le h$, in $\PG(1,q^h)$ with one point of weight $k-2$ and all other points of weight $2$ then $\mathbb{F}_{q^h}$ must has geometric field of linearity $\mathbb{F}_{q^s}$ with $s\mid k-2$, $s>1$, and $s\mid h$.
\end{proposition}

\begin{proof}
	We prove the statement by showing that any linear set satisfying the conditions of Theorems \ref{thr} or \ref{thrm} admits a choice of basis vectors so that it can be obtained from Construction \ref{const}.
	
	If a linear set satisfies the conditions of Theorem \ref{thr} then similar to Lemma \ref{lemreg} we can take $\Sigma=\langle span( \u_1,\v_1,\u_2,\v_2)\rangle_q$ such that $\Pi=span(\langle\u_1-\alpha\v_1\rangle,\langle\u_2-\alpha\v_2\rangle)$. Now mapping the basis vectors $\u_1,\v_1,\u_2,\v_2$ to $ \e_1,\e_2,\e_3,\e_4$ respectively we immediately see that $L$ is obtained from Construction \ref{const} with $k=4,p=1,s=2$. Thus by Lemma \ref{linearity} this set has geometric field of linearity $\mathbb{F}_{q^2}$, and since it can be written as an $\F_{q^2}$-linear set with $q^2+1$ points, we find that it is indeed an $\F_{q^2}$-subline.

		If a linear set satisfies the conditions of Theorem \ref{thrm} then we have $\Sigma=span(H,\ell)$, where $\ell$ is a $\Pi$-line corresponding to a point of weight $2$, so there exist vectors $\w_{k-1}$ and $\w_{k}$ and an element $\alpha\in \F_{q^h}\setminus \F_q$ such that $\ell=span(\langle \w_{k-1}\rangle,\langle \w_{k}\rangle)$ and $\langle \w_{k-1}-\alpha \w_{k}\rangle\in\Pi$. 
		Let $[\F_{q}(\alpha):\F_q]=s$. If $s>2$ then it follows from Remark \ref{srem} and Lemma \ref{lemdiv}, that $H$ can be partitioned by $(s-1)$-dimensional subspaces which form a Desarguesian spread and that we can write $H=span(H_1, \dots, H_r)$ where $H_i=span(\langle \w_{i,1}\rangle,\langle\w_{i,2}\rangle,\ldots,\langle\w_{i,s}\rangle)$ such that $\langle \w_{i,1}-\alpha \w_{i,2}\rangle,\langle \w_{i,2}-\alpha \w_{i,3}\rangle,\dots,\langle \w_{i,s-1}-\alpha \w_{i,s}\rangle\in \Pi$. 
		
	By possibly relabeling the vectors $\w_{i,1}, \w_{i,2},\w_{i,s}$ whose corresponding points span each $H_i$, we can make sure that the set of $r$ vectors $\w_{j,s}$, $j=1,\ldots,r$ are linearly independent. The line joining $\langle \w_{i,s}\rangle$ and $\langle \w_{i+1,s}\rangle$, $i=1,\dots, r-1$ is a $\Pi$ line, say of type $\beta_{i}$ such that $\langle \w_{i,s}-\beta_{i}\w_{i+1,s}\rangle\in\Pi$. It follows that $H\cap \Pi$ equals $span( \Pi_1,\Pi_2,\dots,\Pi_r,\langle \w_{1,s}-\beta_1\w_{2,s}\rangle,\langle \w_{2,s}-\beta_2\w_{3,s}\rangle,\dots,\langle \w_{r-1,s}-\beta_{r-1}\w_{r,s}\rangle)$. It follows that $$\Pi=span( \Pi_1,\Pi_2,\dots,\Pi_r,\langle \w_{1,s}-\beta_1\w_{2,s}\rangle,\langle \w_{2,s}-\beta_2\w_{3,s}\rangle,\dots,$$ $$\langle \w_{r-1,s}-\beta_{r-1}\w_{r,s}\rangle,\langle \e_{rs+1}-\alpha\e_{rs+2}\rangle).$$
	
	Now mapping the basis vectors $\w_{{1,1}},\dots,\w_{{r,s}},\w_{{k-1}},\w_{k}$ to $\e_{1},\dots,\e_{k}$ we see that $\Pi$ and $\Sigma$ are as in Construction \ref{const}. Finally, we know that $\Omega$ can be taken arbitrarily but disjoint from $\Pi$, so we may take $\Omega$ to be $\Omega:=span(\langle \e_{r,s}\rangle,\langle\e_{rs+2}\rangle) .$ Note that the points $\langle \e_{r,s}\rangle$ and $\langle \e_{rs+2}\rangle$ lie in subspaces that are projected onto different points which means that $\Omega$ is indeed disjoint from $\Pi$.
	By Lemma \ref{linearity} this set has geometric field of linearity $\mathbb{F}_{q^s}$.
	
	If $s=2$ in Theorem \ref{thrm}, then by Remark \ref{2rem}, $H$ is partitioned by $\Pi$-lines of type $S_\alpha$ into a Desarguesian spread. Thus we can write $H=span( \ell_1,\dots,\ell_r)$, with $\ell_i=span(\langle \w_{{i,1}}\rangle,\langle\w_{{i,2}}\rangle)$ and $\langle \w_{{i,1}}-\alpha \w_{{i,2}}\rangle\in \Pi$. Arguing as above, the linear set is obtained by Construction \ref{const}.
	By Lemma \ref{linearity} this set has geometric field of linearity $\mathbb{F}_{q^2}$.
\end{proof}

\begin{remark}\label{notmax} It remains to indicate why we do not draw the conclusion that our set has {\em maximum} geometric field of linearity $\F_{q^s}$. First note that the parameter $s$ we find in the proof is determined by the type $S_\alpha$ of the lines corresponding to the points of weight $2$ in $L_U$.

When we take a subline $L=\PG(1,q^{k-2})$ in $\PG(1,q^h)$, then the maximum geometric field of linearity is clearly $\F_{q^{k-2}}$. We can write this set as $L_U$ where $U$ is an $\F_q$-vector space of rank $2k-4$ which is an $\F_{q^{k-2}}$-linear vector space.  Then $\mu=\langle U\rangle$ is a $(2k-5)$-dimension projective subspace of $\PG(2h-1,q)$, partitioned by $(k-3)$-spaces, each corresponding to a point of $L_U$. This forms a Desarguesian $(k-3)$-spread, say $\mathcal{D'}$, of $\mu$. Let $H$ be one of the $(k-3)$-dimensional spread elements of $\mathcal{D'}$.
Now for every proper divisor $s$ of $k-2$, we find a unique $(s-1)$-subspread of $\mathcal{D'}$  (see e.g. \cite{dye,spreads}), and the $\Pi$-lines of type $S_\alpha$ where $[\F_q(\alpha):\F_q]=s$ through a point of $H$ are entirely contained in an element of this $(s-1)$-subspread $\mathcal{D'}$. Vice versa, every line contained in a spread element of $\mathcal{D'}$ is a $\Pi$-line of type $S_\beta$ where $[\F_q(\beta):\F_q]=s$.

This means that if we take a subspace $V$ of $U$ spanned by $H$ and a line contained in one of the induced elements of the $(s-1)$-subspread in $\mu$, we will obtain that $L_U=L_V$ where $L_V$ satisfies the hypothesis of our theorem. But we will draw the conclusion that $L_V$ has geometric field of linearity $\F_{q^s}$, whereas the maximum field of linearity is $\F_{q^{k-2}}$. If we start however from a subspace $V'$ determined by $H$ and any line in a spread element of $\mathcal{D'}$ which is not contained in any of the induced subspreads, then the corresponding field element $\alpha$ generates $\F_{q^{k-2}}$ and our theorem will lead to the conclusion that $L_V'$ has geometric field of linearity $\F_{q^{k-2}}$.

Finally note that if $ls=k-2 \mid h$, an $\F_{q^s}$-linear set of size $q^{k-2}+1$ does not necessarily have maximum geometric field of linearity $\F_{q^{k-2}}$. If a linear set of size $q^{k-2}+1$ has maximum geometric field of linearity $\F_{q^{k-2}}$, then it is necessarily a subline $\PG(1,q^{k-2})$; but not all $\F_{q^s}$-linear sets of size $q^{ls}+1$ are $\F_{q^{k-2}}$-sublines. The easiest case to see this is for $\F_q$-linear sets of size $q^2+1$ in $\PG(1,q^4)$: some of those are $\PG(1,q^2)$-sublines while others are not (see also \cite[page 9]{jena}).
\end{remark}

\begin{corollary} Let $L_U$ be an $\F_q$-linear $i$-club of rank $k$ in $\PG(1,q^h)$, that is, $L_U$ has one point of weight $i$ and all others of weight $1$. If $L_U=L_V$ for some proper subspace $U$ of $V$, then $L_U$ has maximum field of linearity $\F_{q^s}$ for some $s>1$.
\end{corollary}
\begin{proof} Let $W$ be a $(k+1)$-dimensional subspace with $U<W\leq V$. We see that $L_W$ has one point of weight $i$ and all other points of weight $2$, and the statement follows from Proposition \ref{coromain}.
\end{proof}

We are now ready to prove the main theorem.
\begin{theorem}\label{main}
	If $L$ is a linear set of rank $k$, $4\le k\le h$, in $\PG(1,q^h)$ with one point $P$ of weight $k-w\ge 2$ and all other points of the same weight $w\ge 2$ then $L$ has geometric field of linearity $\mathbb{F}_{q^s}$ with $s\mid k-w$, $s>1$, $s\mid h$, and $s\geq w$.
\end{theorem}

\begin{proof} We switch to the representation of linear sets in terms of Desarguesian spreads.
	Recall that the linear set $L=L_U$ can be viewed as the pre-image of a $(k-1)$-dimensional subspace $\sigma=\langle U\rangle$ of $\PG(2h-1,q)$, under the field reduction map $\phi:\PG(1,q^h)\rightarrow\PG(2h-1,q)$. As always, let $\mathcal{D}$ denote the $(h-1)$-spread determined by $\phi$.
	

	Let $\phi(P)\cap\sigma=\sigma_P$, then $\sigma_P$ is a $(k-w-1)$-dimensional subspace of $\PG(2h-1,q)$ since $P$ has weight $k-w$. Let $Q\neq P$ be a point of $L_U$, then $\phi(Q)\cap \sigma=\sigma_Q$ is $(w-1)$-dimensional. Now consider a $(k-w+1)$-dimensional subspace $\sigma'$ of $\sigma$ containing $\sigma_P$, then $\sigma'=\langle V\rangle$ with $V<U$ and $L_U=L_V$.
	Since $\sigma'\cap \phi(P)=\sigma_P$, $P$ has weight $k-w$ in $L_V$; moreover, it follows that $\phi(Q)\cap \sigma'=\sigma'_Q$ is a line, and hence, $Q$ has weight $2$ in $L_V$.  This implies that
	$L_V$ is an $\F_q$-linear set of rank $k-w+2$ with one point of weight $k-w$ and all other points of weight $2$. Hence by Proposition \ref{coromain} we see that $L_V$, and hence also $L_U$, has geometric field of linearity $\F_{q^s}$ with $s\mid k-w$.
	
	The only thing left to prove is that $L_U$ has geometric field of linearity $\F_{q^{s'}}$ for some $s'\geq w$.
	In the proof of Proposition \ref{coromain}, we have deduced that $L_V=L_W$ where $W$ is $\F_{q^s}$-linear, has $\F_q$-dimension $k-w+s$ and $\langle W\rangle$ is spanned by $\sigma_P$ and an $(s-1)$-dimensional subspace $\tau$ of $\phi(Q)$ containing the line $\sigma'_Q$. Recall that $\sigma_Q$ is $(w-1)$-dimensional. Hence, if $\tau$ contains $\sigma_Q$, then $s-1\geq w-1$ and we are done.
	
	So suppose that there is a point of $\sigma_Q$, say $R$, not contained in $\tau$. Consider a line $m$ through $R$ and a point of $\sigma'_Q$. Let $V'$ be a $(k-w+2)$-dimensional vector subspace such that $\langle V'\rangle=span(\sigma_P,m)$. Then $L_U=L_{V'}$, and, $L_{V'}$ has one point of weight $k-w$ and all others of weight $2$. It follows that $L_{V'}$ has geometric field of linearity $\F_{q^{s'}}$ for some $s'\mid k-w$. Furthermore, as before, $m$ is contained in an $(s'-1)$-dimensional subspace $\tau'$ of $\phi(Q)$ such that there exists an $\F_{q^{s'}}$-linear subspace $W'$ of $\F_q$-dimension $k-w+s'$ with $\langle W'\rangle=span(\sigma_P,\tau')$. It follows that $L_{V'}=L_{W'}$, and furthermore, if we put $X=span(W,W')$, then it is clear that $L_X=L_W=L_{W'}=L_{V'}=L_U$.
	

	The $(s-1)$-space $\tau$ through $R$ is an element of the unique Desarguesian $(s-1)$-subspread of $\mathcal{D}$ while the $(s'-1)$-space $\tau'$ through $R$ is an element of the unique Desarguesian $(s'-1)$-subspread of $\mathcal{D}$. Since $\tau\neq\tau'$, we have that $s\neq s'$, and furthermore, $span(\tau,\tau')$ is contained in the unique Desarguersian $(t-1)$-spread of $\mathcal{D}$ where $t=lcm(s,s')$. This in turn implies that $X$ is $\F_{q^t}$-linear so $L_U$ has geometric field of linearity $t$. 
	
	If $t< w$, then there is a point $R'$ of $\sigma_P$, not contained in $span(\tau,\tau')$ and we can repeat the reasoning above to find that $R'$ lies in an element of the Desarguesian $(s''-1)$-subspread $\tau''$ with $s''\neq t$,  $L_U$ can be written as $L_{X'}$ where $X'$ is $\F_{q^{t'}}$-linear with $t'=lcm(t,s'')$. Eventually, we fill find that all points of $\sigma_P$ are contained in an element $\tau''$ of a Desarguesian $(t''-1)$-spread and that $L_U=L_{X''}$ where $X''$ is $\F_{q^{t''}}$-linear. Since $\sigma_P$ is contained in $\tau''$, we find that so $t''-1\geq w-1$ and we are done.
	
	\end{proof}

\section{Conclusion} 

	In this paper we have shown that if a rank $k$ $\mathbb{F}_q$-linear set on a line $\PG(1,q^h)$, $k\leq h$, has one point of weight $k-w$ and all others of weight $w$, then $L$ has geometric field of linearity $\mathbb{F}_{q^s}$ for some $1<s$, $s\mid k-w$, $s\geq w$. 
	 As indicated in Remark \ref{club}, these linear sets can be viewed as $\F_{q^s}$-linear clubs or $\F_{q^s}$-sublines, which are in some sense, the `easiest' types of linear sets.
	\paragraph{Main open problem}
	The larger question of whether the same conclusion can be obtained for all linear sets without points of weight one, remains unsolved. Although we believe that this should indeed be the case, the methods developed in this paper are insufficient to tackle this question. 
	 
	 The case that we have studied in this paper has the advantage that the field of linearity we are looking for (say $\F_{q^s}$) necessarily has $s$ dividing $k-w$, where $k-w$ is the weight of the unique point, different from all the ones of weight $w$ appearing in the set. In general, if we take an $\F_{q^s}$-linear set $L_U$, where $U$ has rank $ls$ and consider a subspace $V$ of $U$ of dimension at least $ls-s+2$, all points of $L_V$ will have weight at least $2$. However, there is no longer an obvious way to deduce the value of $s$ from the weights of the points in $L_V$.

	 \paragraph{A geometric point of view using subspreads}
	Geometrically speaking, the belief we indicated above is equivalent to the following: whenever a subspace $\mu$ meets all elements of a Desarguesian spread $\mathcal{D}$ in either $0$ points or in at least a line, the partition of $\mu$ induced by the elements of $\mathcal{D}$ consists of parts, each of which are contained in some fixed Desarguesian subspread of $\mathcal{D}$ (induced by a subfield). We have shown in this paper that the above statement indeed holds true if the partition consists of one $(k-w-1)$-space and all other parts are $(w-1)$-spaces.
	 
	 \paragraph{The rank of the linear set} We have given examples of linear sets whose rank is not uniquely defined, in the sense that this linear set can be written as $L_U$ and $L_V$ with $U,V$ subspaces of different dimensions. If $L_U=L_V$ where $U<V$, then it is clear that all points of $L_V$ have weight at least $2$. As indicated above, we believe that this should imply that $L_U$ has maximum geometric field of linearity $\F_{q^s}$ for some $s>1$. However, since not every linear set is simple, it does not immediately follow from $L_U=L_V$ with $dim(U)<dim(V)$ that there exists a subspace $W$ with $dim(V)=dim(W)$, $U<W$ and $L_U=L_W$. In other words, it would be interesting to see whether $L_U=L_V$ for $U,V$ subspaces with $dim(V)>dim(U)$ always implies that all points of $L_V$ have weight at least two.

	 \paragraph{A polynomial point of view}
	 Considerable effort has been done in recent years to investigate linear sets from a (linearised) polynomial point of view -- every rank $h$ linear set in 
	 $\PG(1,q^h)$ can be described through such a polynomial as a set determined by the points $\langle(x,f(x))\rangle$, where $x$ ranges in $\F_{q^h}$. For a recent result linking linearised polynomials with the weights of the point sets they define, see \cite{fat}.
	 
	 But as indicated in the introduction (see Results \ref{prop22}), for linear sets of rank $h$ in $\PG(1,q^h)$, it is not only known that if they contain only points of weight at least $2$, the linear set have geometric field of linearity $\F_{q^s}$. They are, in fact, already $\F_{q^s}$-linear themselves.
	While it is possible to describe a linear set of rank $k<h$ as a set $\langle(x,f(x))\rangle$ where $x$ ranges in a subspace of $\F_{q^h}$, it seems that the fact that $x$ does not range in the entire field makes a polynomial approach (much) more complicated. 
	
	In \cite{takats}, the authors study the direction problem for point sets of size less than $q$ and prove a statement which resembles Theorem \ref{th:graph1}. However, we have seen that we cannot hope to extend the result of Proposition \ref{prop22} to all linear sets. This is in part indicated by the fact that in the more general theorem of \cite{takats}, the authors have to take into account two different parameters, $s$ and $t$. These parameters will essentially correspond to the algebraic and geometric field of linearity of the point sets if we look at the directions determined by an $\F_q$-linear map defined on a subspace of $\F_{q^h}$.

\end{document}